\documentclass{article}

\usepackage{arxiv}

\usepackage[utf8]{inputenc} 
\usepackage[T1]{fontenc}    
\usepackage{hyperref}       
\usepackage{url}            
\usepackage{booktabs}       
\usepackage{amsfonts}       
\usepackage{nicefrac}       
\usepackage{microtype}      
\usepackage{lipsum}
\usepackage{graphicx}
\graphicspath{ {./images/} }

\usepackage{amsmath}
\usepackage{comment}
\usepackage{amsthm}
\usepackage{enumerate}
\usepackage{multicol}

\newtheorem{theorem}{Theorem}[section]
\newtheorem{lemma}{Lemma}[section]
\newtheorem{corollary}{Corollary}[section]
\newtheorem{remark}{Remark}[section]

\title{Distance Eigenvalues, Forwarding Indices, and Distance-based Topological Indices of Complement of Two Circulant Networks}

\author{
 John Rafael Macalisang Antalan \\
  Department of Mathematics and Physics\\
  Central Luzon State University (3120)\\
  Science City of Mu\~{n}oz, Nueva Ecija Philippines \\
  \texttt{jrantalan@clsu.edu.ph} \\
   \And
 Francis Joseph Hernandez Campe\~{n}a\\
  Mathematics and Statistics Department\\
  De La Salle University (1004)\\
  2401 Taft Ave., Malate, Manila, Philippines \\
  \texttt{francis.campena@dlsu.edu.ph} \\
 }

\begin{document}

\maketitle
\begin{abstract}
Let $n$ and $a$ be positive integers such that $2\leq a\leq \frac{n}{2}$. In this short note, we compute for the exact value of the distance spectral radius, vertex-forwarding index, and some distance-based topological indices of the complement of circulant networks $C_n(1,a)$ and $C_{m^h}(1,m,m^2,\ldots,m^{h-1})$. For $a\neq \frac{n}{2}$, the circulant $C_n(1,a)$ is called a double loop network while the circulant    $C_{m^h}(1,m,m^2,\ldots,m^{h-1})$ is called the multiplicative circulant network on $m^h$ vertices.
\end{abstract}


\section{Introduction}

Let $\Gamma$ be a graph with vertex set $V(\Gamma)$ and edge set $E(\Gamma)$. For two vertices $v_i$ and $v_j$ in $V(\Gamma)$, the distance between them denoted by $d_{\Gamma}(v_i,v_j)$ is the length of a shortest path between $v_i$ and $v_j$. The distance matrix of $\Gamma$ denoted by $\textbf{D}(\Gamma)$ is the matrix whose $ij-$ entry is $d_{\Gamma}(v_i,v_j)$ if $v_i\neq v_j$; and $0$ otherwise. The \textbf{distance spectral radius} of a graph $\Gamma$ denoted by $\rho(\Gamma)$ refers to the largest eigenvalue of $\textbf{D}(\Gamma)$. 

A graph property related to distance between vertices in a graph is the distance-based topological index. A \textbf{topological index} is a real number associated to a graph which characterizes its topology. It is invariant under graph automorphism. A topological index is said to be distance-based if its computation involves distance between vertices in graphs. Many of the known topological indices have applications in chemical graph theory. For instance, the applications of Wiener index $W$ defined by 
\begin{equation*}
W(\Gamma)=\sum_{\{v_i,v_j\}\subseteq V(\Gamma)}{d_{\Gamma}(v_i,v_j)}
\end{equation*}
which is the oldest distance-based topological index (related to molecular branching) is presented in \cite{Iyer}.

Another graph property that depends on distance between vertices is the concept of graph forwarding index. In order to discuss graph forwarding index, we need to talk first about routing in a graph. In what follows are some of the definitions presented by Xu and Xu in \cite{Xu}.  

A \textbf{routing} $R$ of a graph $\Gamma$ is a set of $n(n-1)$ elementary paths $R(x,y)$ specified for all ordered pairs $(x,y)$ of vertices of $\Gamma$. If each of the paths specified by $R$ is shortest, the routing $R$ is said to be \textbf{minimal}, denoted by $R_m$. If $R(x,y)=R(y,x)$ specified by $R$, that is to say the path $R(y,x)$ is the reverse of the path $R(x,y)$ for all $x,y$, then the routing is \textbf{symmetric}. Finally, the set of all possible routing in a graph $\Gamma$ is denoted by $\mathcal{R}$$(\Gamma)$ and the subset of $\mathcal{R}$$(\Gamma)$ whose elements contains all the minimum routing in a graph $\Gamma$ is denoted by $\mathcal{R}$$_m$$(\Gamma)$. 

Now, let $R\in\mathcal{R}$$(\Gamma)$ and $x\in V(\Gamma)$. The \textbf{load of a vertex} $x$ in $R$ of $\Gamma$ denoted by $\xi_x(\Gamma,R)$ is the number of paths specified by $R$ passing through $x$ and admitting $x$ as an inner vertex. The \textbf{ vertex-forwarding index of $\Gamma$ with respect to $R$}, denoted by $\xi(\Gamma,R)$  is the maximum number of paths of $R$ going through any vertex $x$ in $\Gamma$. Hence
\begin{equation*}
\xi(\Gamma,R)=\mbox{max}\{\xi_x(\Gamma,R):x\in V(\Gamma)\}.
\end{equation*}

As an illustrative example, consider the graph $\Gamma$ shown in Figure \ref{routing}. 

\begin{figure}[ht!]
\begin{center}
\includegraphics[width=0.3\columnwidth]{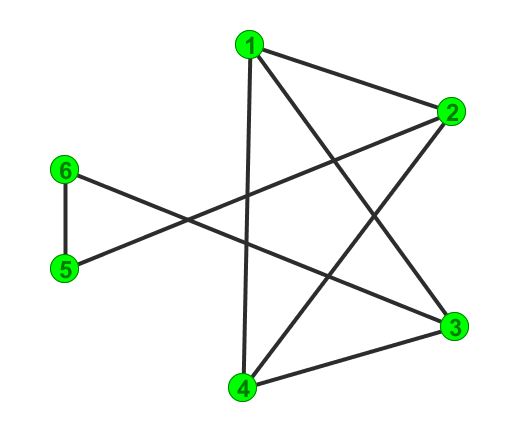}
\caption{{The graph $\Gamma$.
{\label{routing}}%
}}
\end{center}
\end{figure}

The sets 
\begin{center}
$R_1=\{(1,2),(1,3),(1,4),(1,2,5),(1,3,6),(2,1),(2,1,3),(2,4),(2,5),(2,5,6),$
$(3,1),(3,1,2),(3,4),(3,6,5),(3,6),(4,1),(4,2),(4,3),(4,2,5),(4,3,6),(5,2,1),$ $(5,2),(5,6,3),(5,2,4),(5,6),(6,3,1),(6,5,2),(6,3),(6,3,4),(6,5)\}$.
\end{center}
and 
\begin{center}
$R_2=\{(1,4,2),(1,3),(1,4),(1,3,6,5),(1,3,6),(2,1),(2,1,3),(2,4),(2,5),$ $(2,1,3,6), (3,1),(3,1,2),(3,4),(3,4,1,2,5),(3,6),(4,1),(4,2),(4,3),(4,3,6,5),$ $(4,3,6), (5,2,1),(5,6,3,1,2),(5,6,3),(5,2,4),(5,6),(6,3,1),(6,5,2),(6,3),$ $(6,3,4),(6,3,4,2,5)\}$.
\end{center}    

are routing of $\Gamma$. Observe that $R_1$ is a minimal routing while $R_2$ is not. Moreover, the load of vertex $3$ in $R_1$ of $\Gamma$ is $4$, that is $\xi_3(\Gamma,R_1)=4$. While the load of vertex $3$  in $R_2$ of $\Gamma$ is $9$, that is  $\xi_3(\Gamma,R_2)=9$.
\vspace{2mm} 

Finally, it can be verified that the load of each vertex in the routing $R_1$ of $\Gamma$ is given by: 1: 3, 2: 4, 3: 4, 4: 0, 5: 2, 6: 2. Hence, the forwarding index of $\Gamma$ with respect to $R_1$ is $4$. 

On the other hand, the load of each vertex in the routing $R_2$ of $\Gamma$ is given by: 1: 7, 2: 7, 3: 9, 4: 3, 5: 2, 6: 2. Hence, the forwarding index of $\Gamma$ with respect to $R_2$ is $9$.

The \textbf{vertex-forwarding index of $\Gamma$}, denoted by $\xi(\Gamma)$ is the minimum forwarding index over all possible routing of $\Gamma$. In symbol,  
\begin{equation*}
\xi(\Gamma)=\mbox{min}\{\xi(\Gamma,R):R\in {\mathcal{R}}({\Gamma})\}.
\end{equation*}

A similar definition for the edge-forwarding index of a graph $\Gamma$ denoted by $\pi(\Gamma)$ can be made by replacing the word ``vertex'' by ``edge'' in the definitions being stated.  

The concept of graph forwarding indices is applied in network designs.This application was discussed in the works of Chung et al. \cite{Chung}, Heydemann \cite{Hey}, and Xu et al. \cite{Xu}. 

Recently, the exact value of some distance-based topological indices and vertex-forwarding index of some families of circulant graph class were computed. 

Let $G$ be a group and $S$ be a subset of $G\backslash\{e\}$. A graph $\Gamma$ is a \textbf{Cayley graph} of $G$ with connection (or jump) set $S$, written $\Gamma=Cay(G,S)$ if  $V(\Gamma)=G$ and 
$E(\Gamma)=\{\{g,sg\}:g\in G, s\in S\}$. If $G=\langle\mathbb{Z}_n,+_n\rangle$, then the graph $\Gamma=Cay(G,S)=C_n(S)$ is called the \textbf{circulant graph} with connection (or jump) set $S$. Note that for $s$ and $s^{-1}$ in $\mathbb{Z}_n$, we have $\{\{g,s+_n g\}:g\in G\}=\{\{g,s^{-1}+_n g\}:g\in G\}$. Hence, for a circulant graph, we have $S\subseteq \{1,2,\ldots,\lfloor \frac{n+1}{2} \rfloor\}$. 

Circulant graphs can also be defined in terms of their adjacency matrix. In particular, circulant graphs are graphs with circulant adjacency matrix. Recall, an $n\times n$ matrix \textbf{$M$} is said to be \textbf{circulant} if each row in \textbf{$M$} is rotated one element to the right relative to the preceding row.  Figure \ref{circ} shows some examples of circulant graphs.

\begin{figure}[ht!]
\begin{center}
\includegraphics[width=0.22\columnwidth]{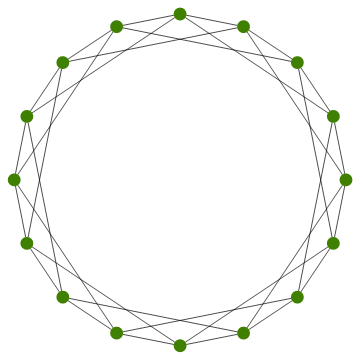}
\includegraphics[width=0.22\columnwidth]{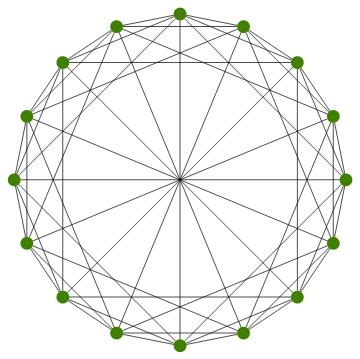}
\includegraphics[width=0.22\columnwidth]{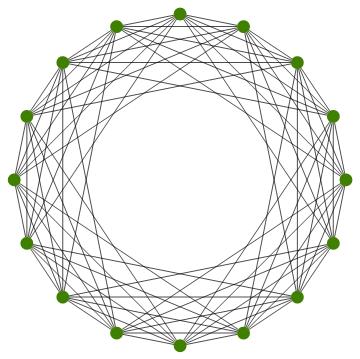}
\caption{The graphs $C_{16}(1,3)$, $C_{16}(1,2,4,8)$ and $C_{16}(1,2,3,4,5)$ respectively.}
\label{circ}
\end{center}
\end{figure}

Circulant graphs have vast applications in different fields of study; some of these fields include telecommunication networking \cite{Bermond}, VLSI (Very-large-scale integration) design \cite{Leighton}, and distributed computing \cite{Mans}. Other applications of circulant graphs are provided in the work of Monakhova \cite{Mona}, and the references therein. 

This short research note is motivated by our previous work \cite{Antalan} and the works of Ali et al. and Lin et al.. In \cite{Antalan}, we were able to determine the distance spectral radius, vertex-forwarding index, and bounds for the edge-forwarding index of the circulant network $C_{m^h}(1,m,m^2,\ldots,m^{h-1})$ where $m$ is odd. Ali et al. \cite{Ali1,Ali2} computed the Wiener, hyper-Wiener, and Schutlz index of circulants $C_n(1,a)$, where $a=2,3,4$, and $5$. While Lin et al. \cite{Liu} computed the exact values of the vertex-forwarding index of circulants with the following connection sets (i) $S=\{1,a\}$ where $a=\frac{n}{2}$, (ii) $S=\{1,2,\ldots, a\}$ where $2\leq a\leq \frac{n}{2}$, and (iii) $S=\{1,a\}$ where $2\leq a<\frac{n}{2}$. They also obtained an upper and lower bound for the edge-forwarding index of the said graphs.    

The main objective of this study is to determine the distance matrix of the complement of circulant networks (i) $C_n(1,a)$ and (ii) $C_{m^h}(1,m,m^2,\ldots,m^{h-1})$.

As a consequence, we have the following secondary objectives:

\begin{enumerate}
    \item Determine the exact value of the distance spectral radius of the complement of circulants (i) and (ii).
    \item Found the exact value of the vertex-forwarding index of the complement of circulants (i) and (ii).
    \item Give an upper bound and a lower bound for the edge-forwarding index for the complement of the circulants (i) and (ii). 
    \item Compute for some well-known distance-based topological indices for the complement of circulants (i) and (ii).
\end{enumerate}

\section{Preliminaries}

In this section, we define some important terms and state some useful results that will be used in the presentation of the main results.

\subsection{The Distance-based Topological Indices}

The distance-based topological indices that we consider in this paper are given in Tables \ref{tabdis}-\ref{tabrtrans}. Table \ref{tabdis} gives some of the most important purely distance-based topological indices. While Table \ref{tabdisdeg} gives some of the most important distance-degree-based topological indices. Recall that the \textbf{degree} of a vertex $v_i$ refers to the number of edges incident to $v_i$. Moreover, we say that a graph $\Gamma$ is \textbf{vertex-regular} if all the vertices in $V(\Gamma)$ have the same degree. 

On the other hand, Table \ref{tabtrans} gives some of the recent transmission-based topological indices of a graph. The \textbf{transmission} of a vertex $v_i$ in $\Gamma$ denoted by $Tr_{\Gamma}(v_i)$ refers to the sum of the distances from $v_i$ to all other vertices in $V(\Gamma)$. In terms of the distance matrix, the transmission of vertex $v_i$ is the sum of the entries of the row indexed by $v_i$ in $\textbf{D}(\Gamma)$. Moreover, we say that a graph $\Gamma$ is \textbf{transmission-regular} if all the vertices in $V(\Gamma)$ have the same transmission.    

Finally, Table \ref{tabrtrans} gives the newly introduced reciprocal transmission-based topological indices. The \textbf{reciprocal transmission} of a vertex $v_i$ in $\Gamma$ denoted by $rs_{\Gamma}(v_i)$ refers to the sum of the reciprocal of the distances from $v_i$ to all other vertices in $V(\Gamma)$. In terms of the distance matrix, the reciprocal transmission of vertex $v_i$ is the sum of the reciprocal of the entries of the row indexed by $v_i$ in $\textbf{D}(\Gamma)$.  

Note that some of the indices in Table \ref{tabtrans} and all the indices in Table \ref{tabrtrans}  first formally appeared in \cite{Rama1}. Note also that the Wiener index is also a transmission-based topological index while the Harary index is also a reciprocal transmission-based topological index.     

\begin{table}[ht!]
\resizebox{0.9\textwidth}{!}{\begin{tabular}{|c|c|c|}
\hline
\textbf{Topological Index} & \textbf{Mathematical Expression} & \textbf{Introduced by and Date Introduced}\\
\hline
Wiener\cite{Wien} & $\displaystyle W(\Gamma)=\sum_{\{v_i,v_j\}\subseteq V(\Gamma)}{d_{\Gamma}(v_i,v_j)}$  & Wiener, 1947\\
\hline
Hyper-Wiener\cite{Ran} & $\displaystyle WW(\Gamma)=\frac{1}{2}\sum_{\{v_i,v_j\}\subseteq V(\Gamma)}{\left[d_{\Gamma}(v_i,v_j)+d_{\Gamma}(v_i,v_j)^2\right]}$  & Randic, 1993 \\
\hline
Harary(\cite{Plav},\cite{Iva}) & $\displaystyle H(\Gamma)=\sum_{\{v_i,v_j\}\subseteq V(\Gamma)}{\frac{1}{d_{\Gamma}(v_i,v_j)}}$ & Plavsic et al. \& Ivanciuc et al., 1993 \\
\hline
\end{tabular}}
\caption{Some distance-based topological indices.} \label{tabdis}
\end{table}

\begin{table}[ht!]
\resizebox{0.9\textwidth}{!}{\begin{tabular}{|c|c|c|}
\hline
\textbf{Topological Index} & \textbf{Mathematical Expression} & \textbf{ Introduced by and Date Introduced}\\
\hline
Schultz\cite{Schultz} & $\displaystyle S(\Gamma)=\sum_{\{v_i,v_j\}\subseteq V(\Gamma)}[deg_{\Gamma}(v_i)+deg_{\Gamma}(v_j)]d_{\Gamma}(v_i,v_j)$  & Schultz, 1989\\
\hline
Gutman\cite{Gut} & $\displaystyle G(\Gamma)=\sum_{\{v_i,v_j\}\subseteq V(\Gamma)}[deg_{\Gamma}(v_i)deg_{\Gamma}(v_j)]d_{\Gamma}(v_i,v_j)$  & Gutman, 1994 \\
\hline
\vtop{\hbox{\strut Additively weighted} \hbox{\strut Harary\cite{Aliz}}} & $\displaystyle H_A(\Gamma)=\sum_{\{v_i,v_j\}\subseteq V(\Gamma)}\frac{deg_{\Gamma}(v_i)+deg_{\Gamma}(v_j)}{d_{\Gamma}(v_i,v_j)}$ & Alizadeh et al., 2013 \\
\hline
\vtop{\hbox{\strut Multiplicatively weighted} \hbox{\strut Harary\cite{Aliz}}} & $\displaystyle H_M(\Gamma)=\sum_{\{v_i,v_j\}\subseteq V(\Gamma)}\frac{deg_{\Gamma}(v_i)\cdot deg_{\Gamma}(v_j)}{d_{\Gamma}(v_i,v_j)}$ & Alizadeh et al., 2013 \\
\hline
\end{tabular}}
\caption{Some distance-degree-based topological indices.} \label{tabdisdeg}
\end{table}

\begin{table}[ht!]
\resizebox{0.9\textwidth}{!}{\begin{tabular}{|c|c|c|}
\hline
\textbf{Topological Index} & \textbf{Mathematical Expression} & \textbf{Introduced by and Date Introduced}\\
\hline
T. geometric-arithmetic\cite{Nara} & $\displaystyle T_{GA}(\Gamma)=\sum_{v_i,v_j\in E(\Gamma)}{\frac{2\sqrt{\sigma(v_i)\sigma(v_j)}}{\sigma(v_i)+\sigma(v_j)}}$  & Narayankar \& Selvan , 2017 \\
\hline
T. sum-connectivity\cite{Sha} & $\displaystyle T_{SC}(\Gamma)=\sum_{v_i,v_j\in E(\Gamma)}{\frac{1}{\sqrt{\sigma(v_i)+\sigma(v_j)}}}$  & Sharafdini \& Reti, 2020\\
\hline
T. arithmetic-geometric\cite{Rama1} & $\displaystyle T_{AG}(
\Gamma)=\sum_{v_i,v_j\in E(\Gamma)}{\frac{\sigma(v_i)+\sigma(v_j)}{2\sqrt{\sigma(v_i)\sigma(v_j)}}}$  & Ramane \& Talwar ,  n.d.\\
\hline
T. atom-bond connectivity\cite{Rama1} & $\displaystyle T_{ABC}(\Gamma)=\sum_{v_i,v_j\in E(\Gamma)}{\sqrt{\frac{\sigma(v_i)+\sigma(v_j)-2}{\sigma(v_i)\sigma(v_j)}}}$  & Ramane \& Talwar ,  n.d.\\
\hline
T. augmented Zagreb\cite{Rama1} & $\displaystyle T_{AZ}(\Gamma)=\sum_{v_i,v_j\in E(\Gamma)}{\left[\frac{\sigma(v_i)\sigma(v_j)}{\sigma(v_i)+\sigma(v_j)-2}\right]^3}$  & Ramane \& Talwar ,  n.d.\\
\hline
\end{tabular}}
\caption{Some transmission-based topological indices.} \label{tabtrans}
\end{table}

\begin{table}[ht!]
\resizebox{0.9\textwidth}{!}{\begin{tabular}{|c|c|c|}
\hline
\textbf{Topological Index} & \textbf{Mathematical Expression} & \textbf{Introduced by and Date Introduced}\\
\hline
R.T. arithmetic-geometric\cite{Rama1} & $\displaystyle RT_{AG}(\Gamma)=\sum_{v_i,v_j\in E(\Gamma)}{\frac{rs(v_i)+rs(v_j)}{2\sqrt{rs(v_i)rs(v_j)}}}$  & Ramane \& Talwar ,  n.d.\\
\hline
R.T. geometric-arithmetic\cite{Rama1} & $\displaystyle RT_{GA}(\Gamma)=\sum_{v_i,v_j\in E(\Gamma)}{\frac{2\sqrt{rs(v_i)rs(v_j)}}{rs(v_i)+rs(v_j)}}$  & Ramane \& Talwar , n.d. \\
\hline
R.T. sum-connectivity\cite{Rama1} & $\displaystyle RT_{SC}(\Gamma)=\sum_{v_i,v_j\in E(\Gamma)}{\frac{1}{\sqrt{rs(v_i)+rs(v_j)}}}$  & Ramane \& Talwar, n.d.\\
\hline
R.T. atom-bond connectivity\cite{Rama1} & $\displaystyle RT_{ABC}(\Gamma)=\sum_{v_i,v_j\in E(\Gamma)}{\sqrt{\frac{rs(v_i)+rs(v_j)-2}{rs(v_i)rs(v_j)}}}$  & Ramane \& Talwar ,  n.d.\\
\hline
R.T. augmented Zagreb\cite{Rama1} & $\displaystyle RT_{AZ}(\Gamma)=\sum_{v_i,v_j\in E(\Gamma)}{\left[\frac{rs(v_i)rs(v_j)}{rs(v_i)+rs(v_j)-2}\right]^3}$  & Ramane \& Talwar ,  n.d.\\
\hline
\end{tabular}}
\caption{Some reciprocal transmission-based topological indices.} \label{tabrtrans}
\end{table}

\subsection{Some Useful Results}

In this subsection, we state some important results that will be used in the presentation of the main results. We begin with the result involving the vertex-regularity of circulant graphs. 

\begin{lemma}
Let $\Gamma=C_n(S)$ be a circulant graph such that $|S|=k$. Then for any vertex $v\in V(\Gamma)$ we have
\begin{equation*}
    deg_{\Gamma}(v)=
    \begin{cases}
    2k-1 & \mbox{if $\frac{n}{2}\in S$}\\
    2k & \mbox{otherwise.}
    \end{cases}
\end{equation*}
\end{lemma}

The next three results show that the distance matrix of the complement of circulant graph is also circulant.

\begin{lemma}[Lin et al.\cite{Liu}]
Let $\Gamma$ be a circulant graph. Then the distance matrix of $\Gamma$ denoted by $\textbf{D}(\Gamma)$ is circulant. 
\label{lemliu}
\end{lemma}

\begin{lemma}[Meijer \cite{Mei}]
Let $\Gamma$ be a circulant graph. Then the complement of $\Gamma$ denoted by $\overline{\Gamma}$ is also circulant. 
\label{lemmei}
\end{lemma}

Combining Lemma \ref{lemliu} and Lemma \ref{lemmei} we have 

\begin{corollary}
Let $\Gamma$ be a circulant graph. The distance matrix of the complement of $\Gamma$ denoted by $\textbf{D}(\overline{\Gamma})$ is circulant.
\label{cordmatccirc}
\end{corollary}

A consequence of Lemmma \ref{lemliu} and Corollary \ref{cordmatccirc} about the transmission of circulant graph $\Gamma$ and its complement $\overline{\Gamma}$ is given next.

\begin{corollary}
Let $\Gamma$ be a circulant graph. Then $\Gamma$ and its complement $\overline{\Gamma}$ are transmission-regular. 
\end{corollary}

The connection of vertex transmission to distance spectral radius, vertex-forwarding index, and bounds for the edge-forwarding index of circulant graph is given in the next series of useful results.

\begin{lemma}[Lin et al.\cite{Liu}]
Let $\Gamma$ be a circulant graph, $v\in V(\Gamma)$, and $\rho(\Gamma)$ be its distance spectral radius. Then
\begin{equation*}
    \rho(\Gamma)=Tr_{\Gamma}(v).
\end{equation*}
\end{lemma}

\begin{lemma}[Lin et al. \cite{Liu}]
If $\Gamma$ is a connected circulant graph of order $n$, then
\begin{equation*}
 \xi(\Gamma)=\xi_m(\Gamma)=\rho(\Gamma)-(n-1).
\end{equation*}
\label{lxi}
\end{lemma} 

\begin{lemma}[Lin et al. \cite{Liu}]
If $\Gamma$ is a connected $r-$regular circulant graph of order $n$, then
\begin{equation*}
\frac{2\rho(\Gamma)}{r}\leq \pi(\Gamma)\leq n+\rho(\Gamma)-(2r-1).
\end{equation*}
\label{lpi}
\end{lemma}

Before going to the last two final result in this section, we recall that the \textbf{diameter} of a graph $\Gamma$ denoted by $diam(\Gamma)$ refers to the maximum distance between any pair of vertices in $V(\Gamma)$.

The two final results of this section are the following:

\begin{lemma}[Gutman et al. \cite{Gut1}]
Let $\Gamma$ be a graph with $n$ number of vertices and $m$ number of edges. If for any two adjacent vertices $u$ and $v$ in $V(\Gamma)$, there exists a third vertex $w$ in $V(\Gamma)$ that is not adjacent to either $u$ or $v$ [also called Property *] then
\begin{enumerate}[i]
    \item $\overline{\Gamma}$ is connected,
    \item the diameter of $\overline{\Gamma}$ is two, and
    \item the Wiener index of $\overline{\Gamma}$satisfies the identity
    \begin{equation*}
        W(\overline{\Gamma})=\binom{n}{2}+m.
    \end{equation*}
\end{enumerate}
\label{lemgut}
\end{lemma}

\begin{lemma}[Gutman et al. \cite{Gutman}]
If $\Gamma$ is a connected graph with $diam(\Gamma)\geq 4$, then $\Gamma$ has property *.
\label{lemdiam4}
\end{lemma}

\begin{remark}
For graph $\Gamma$ satisfying property *, we have 
\begin{equation*}
    d_{\overline{\Gamma}}(u,v)=
    \begin{cases}
    2 & \mbox{if $u$ is adjacent to $v$ in $\Gamma$}\\
    0 & \mbox{if $u=v$}\\
    1 & \mbox{otherwise.}
    \end{cases}
\end{equation*}
\label{rem12}
\end{remark}

\section{Distance Matrix of Complement of Two Circulant Graph}

In this section, we determine the distance matrix of the complement of family of  circulants (i) and (ii). Note that in order to determine the distance matrix of circulant graphs, it is enough to determine the distance of the 0-vertex to all the other vertices of the graph.  We begin by considering the complement of circulant family (i).  

\begin{theorem}
Let $n=2k$. For $k\geq 2$ we have 
\begin{equation*}
    d_{\overline{C_n(1,k)}}(0,v)=
    \begin{cases}
    0 & \mbox{if $v=0$}\\
    2 & \mbox{if $v\in \{1,k,n-1\}$}\\
    1 & \mbox{otherwise.}
    \end{cases}
\end{equation*}
\label{teodc1}
\end{theorem}

\begin{proof}
We prove the theorem by considering two cases. The first case is when $4\leq k\leq 7$. If $4\leq k\leq 7$, we can manually construct the graph $\overline{C_n(1,k)}$ and verify that the result holds.

The second case is when $k>7$. If $k>7$, it follows from Theorem 3.1 in $\cite{Liu}$ (the diameter of $C_n(1,k)$ is $\frac{k}{2}$ if $k$ is even while $\frac{k+1}{2}$ if $k$ is odd) that $diam(C_n(1,k))\geq 4$. Hence, by Lemma \ref{lemdiam4}, Lemma \ref{lemgut} and Remark \ref{rem12}, the result follows.    
\end{proof}

\begin{theorem}
Let $n\geq 8$ and $2\leq a<\frac{n}{2}$. If $v\in \overline{V(C_n(1,a))}$ then \begin{equation*}
    d_{\overline{C_n(1,a)}}(0,v)=
    \begin{cases}
    0 & \mbox{if $v=0$}\\
    2 & \mbox{if $v\in \{1,a,n-a,n-1\}$}\\
    1 & \mbox{otherwise}.
    \end{cases}
\end{equation*} 
With the exception for circulant $\overline{C_8(1,3)}$.
\label{teodimatc2}
\end{theorem}

\begin{proof}
Here we also consider two cases. The first case is when $8\leq n<26$. If $8\leq n<26$, using a computing software, we verified that the result holds except for the complement of $C_8(1,3)$ for $\overline{C_8(1,3)}$ is disconnected. 

For $n\geq 26$, we denote by $\delta(n)=\mbox{min}\{diam(C_n(a,b)):1\leq a<\frac{n}{2},a\neq b\}$. Note that $\delta(n)\leq diam(C_n(1,a))$. Using the Yebra \cite{Yebra} bound for $\delta(n)$, we have
\begin{equation*}
    \left\lceil\frac{\sqrt{2n-1}-1}{2}\right\rceil\leq \delta(n)\leq diam(C_n(1,a)).
\end{equation*}
Note that the expression $\left\lceil\frac{\sqrt{2n-1}-1}{2}\right\rceil$ increases as $n$ increases. So it is enough for us to find the minimum value of $n$ such that $\left\lceil\frac{\sqrt{2n-1}-1}{2}\right\rceil=4$. The solution of the last stated equation is $n=26$. Thus, for $n\geq 26$, we have $diam(C_n(1,a))\geq 4$. Using Lemma \ref{lemdiam4}, Lemma \ref{lemgut} and Remark \ref{rem12}, the result follows. 
\end{proof}

\begin{remark}
If $n=7$ for the family of circulants that was considered in Theorem \ref{teodimatc2}, we have
\begin{equation*}
    d_{\overline{C_7(1,2)}}(0,v)=
    \begin{cases}
    0 & \mbox{if $v=0$}\\
    1 & \mbox{if $v\in \{3,4\}$}\\
    2 & \mbox{if $v\in \{1,6\}$}\\
    3 & \mbox{if $v\in \{2,5\}$},
    \end{cases}
\end{equation*}

\begin{equation*}
    d_{\overline{C_7(1,3)}}(0,v)=
    \begin{cases}
    0 & \mbox{if $v=0$}\\
    1 & \mbox{if $v\in \{2,5\}$}\\
    2 & \mbox{if $v\in \{3,4\}$}\\
    3 & \mbox{if $v\in \{1,6\}$}.
    \end{cases}
\end{equation*}
\label{remdc1}
\end{remark}

Next, we consider the distance matrix of multiplicative circulant graph on $m^h$ vertices. For simplicity, we denote by $\Gamma_{m^h}$ the multiplicative circulant graph $C_{m^h}(1,m,m^2,\ldots,m^{h-1})$.

\begin{theorem}
Let $v\in V(\overline{\Gamma_{m^h}})$. For $m\geq 5$ we have
\begin{equation*}
    d_{\overline{\Gamma_{m^h}}}(0,v)=
    \begin{cases}
    0 & \mbox{if $v=0$}\\
    2 & \mbox{if $v\in\{1,m,m^2,\ldots,m^{h-1},n-m^{h-1},\ldots,n-1\}$}\\
    1 & \mbox{otherwise}.
    \end{cases}
\end{equation*}
\label{teomcd}
\end{theorem}

\begin{proof}
To prove the theorem, we use a result of Tang et al. in \cite{Tang}. Using Theorem 4 in \cite{Tang}, one can verify that 
\begin{equation*}
    diam(\Gamma_{m^h})=
    \begin{cases}
    \frac{h(m-1)+1}{2} & \mbox{if $m$ is even and $h$ is odd}\\
    h\left(\frac{m-1}{2}\right) & \mbox{otherwise}.
    \end{cases}
\end{equation*}
Now, we consider two cases. The first case is when $5\leq m\leq 8$. If $5\leq m\leq 8$, using the diameter formula above reveals that for $h\geq 2$, we have $diam(\Gamma_{m^h})\geq 4$. By Lemma \ref{lemdiam4}, Lemma \ref{lemgut} and Remark \ref{rem12}, the result follows. For circulants $\overline{\Gamma_5}$, $\overline{\Gamma_6}$, $\overline{\Gamma_7}$, and $\overline{\Gamma_8}$, we manually calculated the distance matrix and verified that the result holds. Hence, the result holds for $\overline{\Gamma_{m^h}}$ for $5\leq m\leq 8$. 

The second case is when $m\geq 8$. If $m\geq 8$, using the diameter formula above reveals that for $h\geq 1$, $diam(\Gamma_{m^h})\geq 4$. By Lemma \ref{lemdiam4}, Lemma \ref{lemgut} and Remark \ref{rem12}, the result follows.
\end{proof}

\begin{remark}
The multiplicative circulant graphs $\overline{\Gamma_{2^4}}$, $\overline{\Gamma_{2^5}}$, $\overline{\Gamma_{2^6}}$, $\overline{\Gamma_{3^2}}$, $\overline{\Gamma_{3^3}}$, and $\overline{\Gamma_{4^2}}$ also satisfies the result in Theorem \ref{teomcd}.

For $\overline{\Gamma_{2^3}}$, we have 
\begin{equation*}
    d_{\overline{\Gamma_{2^3}}}(0,v)=
    \begin{cases}
    0 & \mbox{if $v=0$}\\
    1 & \mbox{if $v\in \{3,5\}$}\\
    2 & \mbox{if $v\in \{2,6\}$}\\
    3 & \mbox{if $v\in \{1,7\}$}\\
    4 & \mbox{if $v=4$}.
    \end{cases}
\end{equation*}
\label{remmc23}
\end{remark}

\section{Forwarding Indices and Some Distance-based Topological Indices of $\overline{C_n(1,a)}$}

In this section, we state some of the consequences of Theorem \ref{teodc1}, Theorem \ref{teodimatc2}, and Remark \ref{remdc1}. We begin by considering the distance spectral radius of circulant $\overline{C_n(1,\frac{n}{2})}$. The result follows from the definition of distance spectral radius and Theorem \ref{teodc1}.

\begin{theorem}
Let $n=2k$. For $k\geq 2$ we have
\begin{equation*}
    \rho(\overline{C_n(1,k)})=n+2.
\end{equation*}
\label{teorhoc1}
\end{theorem}

Another consequence of Theorem \ref{teodc1} talks about the reverse transmission of a vertex in $\overline{C_n(1,\frac{n}{2})}$.

\begin{theorem}
Let $n=2k$ where $k\geq 2$, and let $v\in V(\overline{C_n(1,k)})$. Then 
\begin{equation*}
    rs_{\overline{C_n(1,k)}}(v)=\frac{2n-5}{2}.
\end{equation*}
\label{teorsc1}
\end{theorem}

For the vertex-forwarding index and bounds for the edge-forwarding index of circulant $\overline{C_n(1,\frac{n}{2})}$, they can be computed by combining Theorem \ref{teodc1} with Lemma \ref{lxi} and Lemma \ref{lpi}. The results are presented in the next two corollaries.

\begin{corollary}
Let $n=2k$ where $k\geq 2$. Then 
\begin{equation*}
    \xi(\overline{C_n(1,k)})=3.
\end{equation*}
\end{corollary}

\begin{corollary}
Let $n=2k$ where $k\geq 2$. Then 
\begin{equation*}
    \frac{2(n+2)}{n-4}\leq \pi(\overline{C_n(1,k)})\leq 11.
\end{equation*}
\end{corollary}

The next series of results give the exact value of some distance-based topological indices of circulants $\overline{C_n(1,\frac{n}{2})}$. The results follow from the definition of the topological indices combined with Theorem \ref{teodc1}, Theorem \ref{teorhoc1}, Theorem \ref{teorsc1} and the fact that $\overline{C_n(1,\frac{n}{2})}$ is a vertex-regular graph with vertex-regularity $n-4$.  

\begin{corollary}
Let $\overline{\Gamma}=\overline{C_n(1,\frac{n}{2})}$. Then
\begin{multicols}{2}
\begin{enumerate}[(i)]
    \item $W(\overline{\Gamma})=\frac{n(n+2)}{2}$
    \item $S(\overline{\Gamma})=n(n-4)(n+2)$
    \item $G(\overline{\Gamma})=\frac{n(n+2)(n-4)^2}{2}$
    \item $WW(\overline{\Gamma})=\frac{n(n+5)}{2}$
    \item $H(\overline{\Gamma})=\frac{n(2n-5)}{4}$
    \item $H_A(\overline{\Gamma})=\frac{n(n-4)(2n-5)}{2}$
    \item $H_M(\overline{\Gamma})=\frac{n(2n-5)(n-4)^2}{4}$
    \item $T_{AG}(\overline{\Gamma})=\frac{n(n-4)}{2}$
    \item $T_{GA}(\overline{\Gamma})=\frac{n(n-4)}{2}$
    \item $T_{SC}(\overline{\Gamma})=\frac{n(n-4)}{2\sqrt{2}\sqrt{n+2}}$
    \item $T_{ABC}(\overline{\Gamma})=\frac{n(n-4)\sqrt{n+1}}{\sqrt{2}(n+2)}$
    \item $T_{AZ}(\overline{\Gamma})=\frac{n(n-4)(n+2)^6}{16(n+1)^3}$
    \item $RT_{AG}(\overline{\Gamma})=\frac{n(n-4)}{2}$
    \item $RT_{GA}(\overline{\Gamma})=\frac{n(n-4)}{2}$
    \item $RT_{SC}(\overline{\Gamma})=\frac{n(n-4)}{2\sqrt{2n-5}}$
    \item $RT_{ABC}(\overline{\Gamma})=\frac{n(n-4)\sqrt{2n-7}}{2n-5}$
    \item $RT_{AZ}(\overline{\Gamma})=\frac{n(n-4)(2n-5)^6}{128(2n-7)^3}$.
\end{enumerate}
\end{multicols}
\end{corollary}

Now, we consider the distance spectral radius of circulants $\overline{C_7(1,2)}$ and $\overline{C_7(1,3)}$ . The result follows from the definition of distance spectral radius and Remark \ref{remdc1}.

\begin{theorem}
Let $\overline{\Gamma_1}=\overline{C_7(1,2)}$ and $\overline{\Gamma_2}=\overline{C_7(1,3)}$. Then
\begin{equation*}
    \rho(\overline{\Gamma_1})=\rho(\overline{\Gamma_2})=12.
\end{equation*}
\label{teorhoc2a}
\end{theorem}

Another consequence of Remark \ref{remdc1} talks about the reverse transmission of a vertex in $\overline{C_7(1,2)}$ and $\overline{C_7(1,3)}$.

\begin{theorem}
Let $\overline{\Gamma_1}=\overline{C_7(1,2)}$ and $\overline{\Gamma_2}=\overline{C_7(1,3)}$. Then  
\begin{equation*}
    rs_{\overline{\Gamma_1}}(v)=rs_{\overline{\Gamma_2}}(v)=\frac{11}{3}.
\end{equation*}
\label{teorsc2a}
\end{theorem}

For the vertex-forwarding index and bounds for the edge-forwarding index of circulant $\overline{C_7(1,2)}$ and $\overline{C_7(1,3)}$, they can be computed by combining Remark \ref{remdc1} with Lemma \ref{lxi} and Lemma \ref{lpi}. The results are presented in the next two corollaries.

\begin{corollary}
Let $\overline{\Gamma_1}=\overline{C_7(1,2)}$ and $\overline{\Gamma_2}=\overline{C_7(1,3)}$. Then  
\begin{equation*}
    \xi(\overline{\Gamma_1})=\xi(\overline{\Gamma_2})=6.
\end{equation*}
\end{corollary}

\begin{corollary}
Let $\overline{\Gamma_1}=\overline{C_7(1,2)}$ and $\overline{\Gamma_2}=\overline{C_7(1,3)}$. Then  
\begin{equation*}
    12\leq \pi(\overline{\Gamma_1})=\pi(\overline{\Gamma_2})\leq 16.
\end{equation*}
\end{corollary}

The next series of results give the exact values of some distance-based topological indices of circulants $\overline{C_7(1,2)}$ and $\overline{C_7(1,3)}$. The results follow from the definition of the topological indices combined with Remark \ref{remdc1}, Theorem \ref{teorhoc2a}, Theorem \ref{teorsc2a} and the fact that the two graphs are vertex-regular graph with vertex-regularity $2$.  

\begin{corollary}
Let $\overline{\Gamma}$ denote either $\overline{C_7(1,2)}$ or $\overline{C_7(1,3)}$.. Then
\begin{multicols}{3}
\begin{enumerate}[(i)]
    \item $W(\overline{\Gamma})=42$
    \item $S(\overline{\Gamma})=168$
    \item $G(\overline{\Gamma})=168$
    \item $WW(\overline{\Gamma})=70$
    \item $H(\overline{\Gamma})=\frac{77}{6}$
    \item $H_A(\overline{\Gamma})=\frac{154}{3}$
    \item $H_M(\overline{\Gamma})=\frac{154}{3}$
    \item $T_{AG}(\overline{\Gamma})=7$
    \item $T_{GA}(\overline{\Gamma})=7$
    \item $T_{SC}(\overline{\Gamma})=\frac{7\sqrt{6}}{12}$
    \item $T_{ABC}(\overline{\Gamma})=\frac{7\sqrt{22}}{12}$
    \item $T_{AZ}(\overline{\Gamma})=\frac{2\ 612\ 736}{1\ 331}$
    \item $RT_{AG}(\overline{\Gamma})=7$
    \item $RT_{GA}(\overline{\Gamma})=7$
    \item $RT_{SC}(\overline{\Gamma})=\frac{7\sqrt{66}}{22}$
    \item $RT_{ABC}(\overline{\Gamma})=\frac{28\sqrt{3}}{11}$
    \item $RT_{AZ}(\overline{\Gamma})=\frac{12\ 400\ 927}{110\ 592}$.
\end{enumerate}
\end{multicols}
\end{corollary}

Finally, we consider the circulant $\overline{C_n(1,a)}$ where $2\leq a<\frac{n}{2}$. We first determine its distance spectral radius. The result follows from the definition of distance spectral radius and Theorem \ref{teodimatc2}. 

\begin{theorem}
Let $\overline{\Gamma}=\overline{C_n(1,a)}$ where $2\leq a<\frac{n}{2}$. Then
\begin{equation*}
    \rho(\overline{\Gamma})=n+3.
\end{equation*}
\label{teorhoc2b}
\end{theorem}

Another consequence of Theorem \ref{teodimatc2} talks about the reverse transmission of a vertex in $\overline{C_n(1,a)}$ where $2\leq a<\frac{n}{2}$.

\begin{theorem}
Let $\overline{\Gamma}=\overline{C_n(1,a)}$ where $2\leq a<\frac{n}{2}$ and $v\in V(\overline{\Gamma})$. Then  
\begin{equation*}
    rs_{\overline{\Gamma}}(v)=n-3.
\end{equation*}
\label{teorsc2b}
\end{theorem}

For the vertex-forwarding index and bounds for the edge-forwarding index of circulant $\overline{C_n(1,a)}$ where $2\leq a<\frac{n}{2}$, they can be computed by combining Theorem \ref{teodimatc2} with Lemma \ref{lxi} and Lemma \ref{lpi}. The results are presented in the next two corollaries.

\begin{corollary}
Let $\overline{\Gamma}=\overline{C_n(1,a)}$ where $2\leq a<\frac{n}{2}$. Then  
\begin{equation*}
    \xi(\overline{\Gamma})=4.
\end{equation*}
\end{corollary}

\begin{corollary}
Let $\overline{\Gamma}=\overline{C_n(1,a)}$ where $2\leq a<\frac{n}{2}$. Then  
\begin{equation*}
    \frac{2(n+3)}{n-5}\leq \pi(\overline{\Gamma})\leq 14.
\end{equation*}
\end{corollary}

The next series of results give the exact values of some distance-based topological indices of circulants $\overline{C_n(1,a)}$ where $2\leq a<\frac{n}{2}$. The results follow from the definition of the topological indices combined with Theorem \ref{teodimatc2}, Theorem \ref{teorhoc2b}, Theorem \ref{teorsc2b} and the fact that the graph is vertex-regular graph with vertex-regularity $n-5$.  

\begin{corollary}
Let $\overline{\Gamma}=\overline{C_n(1,a)}$ where $2\leq a<\frac{n}{2}$. Then
\begin{multicols}{2}
\begin{enumerate}[(i)]
    \item $W(\overline{\Gamma})=\frac{n(n+3)}{2}$
    \item $S(\overline{\Gamma})=n(n-5)(n+3)$
    \item $G(\overline{\Gamma})=\frac{n(n+3)(n-5)^2}{2}$
    \item $WW(\overline{\Gamma})=\frac{n(2n+14)}{4}$
    \item $H(\overline{\Gamma})=\frac{n(n-3)}{2}$
    \item $H_A(\overline{\Gamma})=n(n-5)(n-3)$
    \item $H_M(\overline{\Gamma})=\frac{n(n-3)(n-5)^2}{2}$
    \item $T_{AG}(\overline{\Gamma})=\frac{n(n-5)}{2}$
    \item $T_{GA}(\overline{\Gamma})=\frac{n(n-5)}{2}$
    \item $T_{SC}(\overline{\Gamma})=\frac{n(n-5)}{2\sqrt{2}\sqrt{n+3}}$
    \item $T_{ABC}(\overline{\Gamma})=\frac{n(n-5)\sqrt{n+2}}{\sqrt{2}(n+3)}$
    \item $T_{AZ}(\overline{\Gamma})=\frac{n(n-5)(n+3)^6}{16(n+2)^3}$
    \item $RT_{AG}(\overline{\Gamma})=\frac{n(n-5)}{2}$
    \item $RT_{GA}(\overline{\Gamma})=\frac{n(n-5)}{2}$
    \item $RT_{SC}(\overline{\Gamma})=\frac{n(n-5)}{2\sqrt{2}\sqrt{n-3}}$
    \item $RT_{ABC}(\overline{\Gamma})=\frac{n(n-5)\sqrt{n-4}}{\sqrt{2}(n-3)}$
    \item $RT_{AZ}(\overline{\Gamma})=\frac{n(n-5)(n-3)^6}{16(n-4)^3}$.
\end{enumerate}
\end{multicols}
\end{corollary}

\section{Forwarding Indices and Some Distance-based Topological Indices of $\overline{C_{m^h}(1,m,m^2,\ldots,m^{h-1})}$}

In this section, we state some of the consequences of Theorem \ref{teomcd} and Remark \ref{remmc23}. We begin by considering the distance spectral radius of circulant $\overline{C_{2^3}(1,2,2^2)}$. The result follows from the definition of distance spectral radius and Remark \ref{remmc23}.

\begin{theorem}
Let $\overline{\Gamma_{2^3}}=\overline{C_{2^3}(1,2,2^2)}$. Then
\begin{equation*}
    \rho(\overline{\Gamma_{2^3}})=16.
\end{equation*}
\label{teorhomc23}
\end{theorem}

Another consequence of Remark \ref{remmc23} talks about the reverse transmission of a vertex in $\overline{C_{2^3}(1,2,2^2)}$.

\begin{theorem}
Let $\overline{\Gamma_{2^3}}=\overline{C_{2^3}(1,2,2^2)}$ and let $v\in V(\overline{\Gamma})$. Then  
\begin{equation*}
    rs_{\overline{\Gamma_{2^3}}}(v)=\frac{47}{12}.
\end{equation*}
\label{teortmc23}
\end{theorem}

For the vertex-forwarding index and bounds for the edge-forwarding index of circulant $\overline{C_{2^3}(1,2,2^2)}$, it can be computed by combining Remark \ref{remmc23} with Lemma \ref{lxi} and Lemma \ref{lpi}. The results are presented in the next two corollaries.

\begin{corollary}
Let $\overline{\Gamma_{2^3}}=\overline{C_{2^3}(1,2,2^2)}$. Then  
\begin{equation*}
    \xi(\overline{\Gamma_{2^3}})=9.
\end{equation*}
\end{corollary}

\begin{corollary}
Let $\overline{\Gamma_{2^3}}=\overline{C_{2^3}(1,2,2^2)}$. Then  
\begin{equation*}
    16\leq \pi(\overline{\Gamma_{2^3}})\leq 21.
\end{equation*}
\end{corollary}

The next series of results give the exact values of some distance-based topological indices of circulant $\overline{C_{2^3}(1,2,2^2)}$. The results follow from the definition of the topological indices combined with Remark \ref{remmc23}, Theorem \ref{teorhomc23}, Theorem \ref{teortmc23} and the fact that the graph is vertex-regular graph with vertex-regularity $2$.  

\begin{corollary}
Let $\overline{\Gamma_{2^3}}=\overline{C_{2^3}(1,2,2^2)}$. Then
\begin{multicols}{3}
\begin{enumerate}[(i)]
    \item $W(\overline{\Gamma_{2^3}})=64$
    \item $S(\overline{\Gamma_{2^3}})=256$
    \item $G(\overline{\Gamma_{2^3}})=256$
    \item $WW(\overline{\Gamma_{2^3}})=120$
    \item $H(\overline{\Gamma_{2^3}})=\frac{47}{3}$
    \item $H_A(\overline{\Gamma_{2^3}})=\frac{188}{3}$
    \item $H_M(\overline{\Gamma_{2^3}})=\frac{188}{3}$
    \item $T_{AG}(\overline{\Gamma_{2^3}})=8$
    \item $T_{GA}(\overline{\Gamma_{2^3}})=8$
    \item $T_{SC}(\overline{\Gamma_{2^3}})=\sqrt{2}$
    \item $T_{ABC}(\overline{\Gamma_{2^3}})=\frac{\sqrt{30}}{2}$
    \item $T_{AZ}(\overline{\Gamma_{2^3}})=\frac{16\ 777\ 216}{3\ 375}$
    \item $RT_{AG}(\overline{\Gamma_{2^3}})=8$
    \item $RT_{GA}(\overline{\Gamma_{2^3}})=8$
    \item $RT_{SC}(\overline{\Gamma_{2^3}})=\frac{8\sqrt{282}}{47}$
    \item $RT_{ABC}(\overline{\Gamma_{2^3}})=\frac{16\sqrt{210}}{47}$
    \item $RT_{AZ}(\overline{\Gamma_{2^3}})=\frac{10\ 779\ 215\ 329}{74\ 088\ 000}$.
\end{enumerate}
\end{multicols}
\end{corollary}

Next, we consider the distance spectral radius of circulant $\overline{C_{2^h}(1,2,2^2,\ldots,2^{h-1})}$. The result follows from the definition of distance spectral radius and Theorem \ref{teomcd}.

\begin{theorem}
Let $\overline{\Gamma_{2^h}}=\overline{C_{2^h}(1,2,2^2,\ldots,2^{h-1})}$. Then
\begin{equation*}
    \rho(\overline{\Gamma_{2^h}})=n+2h-2.
\end{equation*}
\label{teorhomc2h}
\end{theorem}

Another consequence of Theorem \ref{teomcd} talks about the reverse transmission of a vertex in $\overline{C_{2^h}(1,2,2^2,\ldots,2^{h-1})}$.

\begin{theorem}
Let $\overline{\Gamma_{2^h}}=\overline{C_{2^h}(1,2,2^2,\ldots,2^{h-1})}$ and let $v\in V(\overline{\Gamma_{2^h}})$. Then  
\begin{equation*}
    rs_{\overline{\Gamma_{2^h}}}(v)=n-h-\frac{1}{2}.
\end{equation*}
\label{teortmc2h}
\end{theorem}

For the vertex-forwarding index and bounds for the edge-forwarding index of circulant $\overline{C_{2^h}(1,2,2^2,\ldots,2^{h-1})}$, it can be computed by combining Theorem \ref{teomcd} with Lemma \ref{lxi} and Lemma \ref{lpi}. The results are presented in the next two corollaries.

\begin{corollary}
Let $\overline{\Gamma_{2^h}}=\overline{C_{2^h}(1,2,2^2,\ldots,2^{h-1})}$. Then  
\begin{equation*}
    \xi(\overline{\Gamma_{2^h}})=2h-1.
\end{equation*}
\end{corollary}

\begin{corollary}
Let $\overline{\Gamma_{2^h}}=\overline{C_{2^h}(1,2,2^2,\ldots,2^{h-1})}$. Then  
\begin{equation*}
    \frac{2(n+2h-2)}{n-2h}\leq \pi(\overline{\Gamma_{2^h}})\leq 6h-1.
\end{equation*}
\end{corollary}

The next series of results give the exact values of some distance-based topological indices of circulant $\overline{C_{2^h}(1,2,2^2,\ldots,2^{h-1})}$. The results follow from the definition of the topological indices combined with Theorem \ref{teomcd}, Theorem \ref{teorhomc2h}, Theorem \ref{teortmc2h} and the fact that the graph is vertex-regular graph with vertex-regularity $n-2h$.  

\begin{corollary}
Let $\overline{\Gamma_{2^h}}=\overline{C_{2^h}(1,2,2^2,\ldots,2^{h-1})}$. Then
\begin{multicols}{2}
\begin{enumerate}[(i)]
    \item $W(\overline{\Gamma_{2^h}})=\frac{n(n+2h-2)}{2}$
    \item $S(\overline{\Gamma_{2^h}})=n(n-2h)(n+2h-2)$
    \item $G(\overline{\Gamma_{2^h}})=\frac{n(n+2h-2)(n-2h)^2}{2}$
    \item $WW(\overline{\Gamma_{2^h}})=\frac{n(2n+8h-6)}{4}$
    \item $H(\overline{\Gamma_{2^h}})=\frac{n(n-h-\frac{1}{2})}{2}$
    \item $H_A(\overline{\Gamma_{2^h}})=n(n-2h)(n-h-\frac{1}{2})$
    \item $H_M(\overline{\Gamma_{2^h}})=\frac{n(n-h-\frac{1}{2})(n-2h)^2}{2}$
    \item $T_{AG}(\overline{\Gamma_{2^h}})=\frac{n(n-2h)}{2}$
    \item $T_{GA}(\overline{\Gamma_{2^h}})=\frac{n(n-2h)}{2}$
    \item $T_{SC}(\overline{\Gamma_{2^h}})=\frac{n(n-2h)}{2\sqrt{2}\sqrt{n+2h-2}}$
    \item $T_{ABC}(\overline{\Gamma_{2^h}})=\frac{n(n-2h)\sqrt{n+2h-3}}{\sqrt{2}(n+2h-2)}$
    \item $T_{AZ}(\overline{\Gamma_{2^h}})=\frac{n(n-2h)(n+2h-2)^6}{16(n+2h-3)^3}$
    \item $RT_{AG}(\overline{\Gamma_{2^h}})=\frac{n(n-2h)}{2}$
    \item $RT_{GA}(\overline{\Gamma_{2^h}})=\frac{n(n-2h)}{2}$
    \item $RT_{SC}(\overline{\Gamma_{2^h}})=\frac{n(n-2h)}{2\sqrt{2n-2h-1}}$
    \item $RT_{ABC}(\overline{\Gamma_{2^h}})=\frac{n(n-2h)\sqrt{2n-2h-3}}{2n-2h-1}$
    \item $RT_{AZ}(\overline{\Gamma_{2^h}})=\frac{n(2h-n)(1+2h-2n)^6}{128(3+2h-2n)^3}$.
\end{enumerate}
\end{multicols}
\end{corollary}

Finally, we consider the circulant $\overline{C_{m^h}(1,m,m^2,\ldots,m^{h-1})}$ where $m\geq 3$. We begin by determining its distance spectral radius. The result follows from the definition of distance spectral radius and Theorem \ref{teomcd}.

\begin{theorem}
Let $\overline{\Gamma_{m^h}}=\overline{C_{m^h}(1,m,m^2,\ldots,m^{h-1})}$ where $m\geq 3$. Then
\begin{equation*}
    \rho(\overline{\Gamma_{m^h}})=n+2h-1.
\end{equation*}
\label{teorhomcmh}
\end{theorem}

Another consequence of Theorem \ref{teomcd} talks about the reverse transmission of a vertex in $\overline{C_{m^h}(1,m,m^2,\ldots,m^{h-1})}$ where $m\geq 3$.

\begin{theorem}
 Let $\overline{\Gamma_{m^h}}=\overline{C_{m^h}(1,m,m^2,\ldots,m^{h-1})}$ where $m\geq 3$ and let $v\in V(\overline{\Gamma_{m^h}})$. Then  
\begin{equation*}
    rs_{\overline{\Gamma_{m^h}}}(v)=n-h-1.
\end{equation*}
\label{teortmcmh}
\end{theorem}

For the vertex-forwarding index and bounds for the edge-forwarding index of circulant $\overline{C_{m^h}(1,m,m^2,\ldots,m^{h-1})}$ where $m\geq 3$, it can be computed by combining Theorem \ref{teomcd} with Lemma \ref{lxi} and Lemma \ref{lpi}. The results are presented in the next two corollaries.

\begin{corollary}
Let $m\geq 3$ and $\overline{\Gamma_{m^h}}=\overline{C_{m^h}(1,m,m^2,\ldots,m^{h-1})}$. Then  
\begin{equation*}
    \xi(\overline{\Gamma_{m^h}})=2h.
\end{equation*}
\end{corollary}

\begin{corollary}
Let $m\geq 3$ and $\overline{\Gamma_{m^h}}=\overline{C_{m^h}(1,m,m^2,\ldots,m^{h-1})}$. Then  
\begin{equation*}
    \frac{2(n+2h-1)}{n-2h-1}\leq \pi(\overline{\Gamma_{m^h}})\leq 6h+2.
\end{equation*}
\end{corollary}

The next series of results give the exact values of some distance-based topological indices of circulant $\overline{C_{m^h}(1,m,m^2,\ldots,m^{h-1})}$ where $m\geq 3$. The results follow from the definition of the topological indices combined with Theorem \ref{teomcd}, Theorem \ref{teorhomcmh}, Theorem \ref{teortmcmh} and the fact that the graph is vertex-regular graph with vertex-regularity $n-2h-1$.  

\begin{corollary}
Let $\overline{\Gamma_{m^h}}=\overline{C_{m^h}(1,m,m^2,\ldots,m^{h-1})}$ where $m\geq 3$. Then
\begin{multicols}{2}
\begin{enumerate}[(i)]
    \item $W(\overline{\Gamma_{m^h}})=\frac{n(n+2h-1)}{2}$
    \item $S(\overline{\Gamma_{m^h}})=n(n-2h-1)(n+2h-1)$
    \item $G(\overline{\Gamma_{m^h}})=\frac{n(n+2h-1)(n-2h-1)^2}{2}$
    \item $WW(\overline{\Gamma_{m^h}})=\frac{n(n+4h-1)}{2}$
    \item $H(\overline{\Gamma_{m^h}})=\frac{n(n-h-1)}{2}$
    \item $H_A(\overline{\Gamma_{m^h}})=n(n-2h-1)(n-h-1)$
    \item $H_M(\overline{\Gamma_{m^h}})=\frac{n(n-h-1)(n-2h-1)^2}{2}$
    \item $T_{AG}(\overline{\Gamma_{m^h}})=\frac{n(n-2h-1)}{2}$
    \item $T_{GA}(\overline{\Gamma_{m^h}})=\frac{n(n-2h-1)}{2}$
    \item $T_{SC}(\overline{\Gamma_{m^h}})=\frac{n(n-2h-1)}{2\sqrt{2}\sqrt{n+2h-1}}$
    \item $T_{ABC}(\overline{\Gamma_{m^h}})=\frac{n(n-2h-1)\sqrt{\frac{n}{2}+h-1}}{\sqrt{2}(n+2h-1)}$
    \item $T_{AZ}(\overline{\Gamma_{m^h}})=\frac{n(n-2h-1)(n+2h-1)^6}{16(n+2h-2)^3}$
    \item $RT_{AG}(\overline{\Gamma_{m^h}})=\frac{n(n-2h-1)}{2}$
    \item $RT_{GA}(\overline{\Gamma_{m^h}})=\frac{n(n-2h-1)}{2}$
    \item $RT_{SC}(\overline{\Gamma_{m^h}})=\frac{n(n-2h-1)}{2\sqrt{2}\sqrt{n-h-1}}$
    \item $RT_{ABC}(\overline{\Gamma_{m^h}})=\frac{n(n-2h-1)\sqrt{n-h-2}}{\sqrt{2}(n-h-1)}$
    \item $RT_{AZ}(\overline{\Gamma_{m^h}})=\frac{n(1+2h-n)(1+h-n)^6}{16(2+h-n)^3}$.
\end{enumerate}
\end{multicols}
\end{corollary}

\section{Conclusion}

In this research note, we were able to determine the distance matrix of the connected complement of the circulants $C_n(1,a)$ where $2\leq a\leq \frac{n}{2}$ and $C_{m^h}(1,m,m^2,\ldots,m^{h-1})$ where $m\geq 2$. As a consequence, we were able to compute for the distance spectral radius, vertex-forwarding index, and some distance-based topological indices of the connected circulants $\overline{C_n(1,a)}$ and $\overline{C_{m^h}(1,m,m^2,\ldots,m^{h-1})}$. As a possible research problem, we recommend the determination of the distance matrix, as well as the distance spectral radius, vertex-forwarding index and some distance-based topological indices of the connected complement of the circulant $C_n(1,2,\ldots,a)$, where $3\leq a\leq \frac{n}{2}$.   

\section*{Acknowledgements}
The creation of this paper and the research behind it will not be possible without the support of the Philippines' Department of Science and Technology Science Education Institute, Central Luzon State University, and De La Salle University.

\end{document}